\documentclass{article}
\usepackage{mathrsfs, amsmath, amsthm}
\usepackage{graphicx, color}
\usepackage{caption, subcaption}
\usepackage{array}
\usepackage{cases}
\makeatletter
\def\th@plain{%
  \upshape 
}
\makeatother

\makeatletter
\renewenvironment{proof}[1][\proofname]{\par
  \pushQED{\qed}%
  \normalfont \topsep6\p@\@plus6\p@\relax
  \trivlist
  \item[\hskip\labelsep
        \bfseries
    #1\@addpunct{.}]\ignorespaces
}{%
  \popQED\endtrivlist\@endpefalse
}
\makeatother

\newtheorem{theorem}{Theorem}
\numberwithin{theorem}{section}
\newtheorem{lemma}{Lemma}
\newtheorem{corollary}{Corollary}

\newtheorem*{conjecture*}{Conjecture}

\theoremstyle{definition}

\newtheorem{remark}{Remark}

\usepackage[left=25mm,top=25mm,bottom=25mm,right=25mm]{geometry}
\usepackage[T1]{fontenc}
\usepackage[varg]{txfonts}

\usepackage{enumitem}
\usepackage[square, numbers, sort&compress]{natbib}

\ifx\pdfoutput\undefined
 \usepackage[dvipdfm,%
  bookmarks=true,%
  bookmarksnumbered=true, 
  bookmarksopen=true, 
  plainpages=false,%
  pdfpagelabels,%
  colorlinks=true, 
  linkcolor=blue, 
  citecolor=blue,%
  hyperindex=true,
  urlcolor=black]{hyperref}
\else
 \usepackage[pdftex,%
  bookmarks=true,%
  bookmarksnumbered=true, 
  bookmarksopen=true, 
  plainpages=false,%
  pdfpagelabels,%
  colorlinks=true, 
  linkcolor=blue, 
  citecolor=blue,%
  anchorcolor=green,
  urlcolor= blue,
  breaklinks=true,
  hyperindex=true]{hyperref}
\fi

\renewcommand{\figurename}{Fig.}

\newcounter{Hcase}
\newcounter{Hclaim}

\newcommand{\etal}{et~al.\ }

\newcommand{\eg}{e.g.,\ }

\def\int(#1){\mathrm{int}(#1)}
\def\ext(#1){\mathrm{ext}(#1)}
\def\Int(#1){\mathrm{Int}(#1)}
\def\Ext(#1){\mathrm{Ext}(#1)}
\def\ad(#1){\mathrm{ad}(#1)}
\def\mad(#1){\mathrm{mad}(#1)}
\def\la(#1){\mathrm{la}(#1)}

\newcommand{\girth}{\mathrm{girth}}
\begin{document}
\title{Light subgraphs in graphs with average degree at most four}
\author{Tao Wang\,\textsuperscript{a, b, }\footnote{{\tt Corresponding
author: wangtao@henu.edu.cn; iwangtao8@gmail.com}}\\
{\small \textsuperscript{a}Institute of Applied Mathematics}\\
{\small Henan University, Kaifeng, 475004, P. R. China}\\
{\small \textsuperscript{b}School of Mathematics and Statistics}\\
{\small Henan University, Kaifeng, 475004, P. R. China}}
\date{January 12, 2016}
\maketitle
\begin{abstract}
A graph $H$ is said to be {\em light} in a family $\mathfrak{G}$ of graphs if at least one member of $\mathfrak{G}$ contains a copy of $H$ and there exists an integer $\lambda(H, \mathfrak{G})$ such that each member $G$ of $\mathfrak{G}$ with a copy of $H$ also has a copy $K$ of $H$ such that $\deg_{G}(v) \leq \lambda(H, \mathfrak{G})$ for all $v \in V(K)$. In this paper, we study the light graphs in the class of graphs with small average degree, including the plane graphs with some restrictions on girth. 

We proved that: 
\begin{enumerate}
\item If $G$ is a graph with $\delta(G) = 2$, average degree less than $\frac{14}{5}$ and without $(2, 2, \infty)$-triangles, then $G$ has one of the following configurations: a $(2, 2, 13^{-}, 2)$-path, a $(2, 3^{-}, 3^{-})$-path and a $(4; 2, 2, 2, 3^{-})$-star. 
\item If $G$ is a plane graph with $\delta(G) \geq 2$ and face size at least $7$, then $G$ has a $(2, 2, 5^{-})$-path, or a $(2, 5^{-}, 2)$-path or a $(3, 3, 2, 3)$-path. 
\item If $G$ is a graph with $\delta(G) = 2$, average degree less than $3$ and without $(2, 2, \infty)$-triangles, then $G$ has one of the following configurations: a $(2, 3^{-}, 3^{-})$-path, a $(2, 2, \infty, 2)$-path, a $(2, 2, 4, 3)$-path, a $(4; 2, 2, 2, 6^{-})$-star, a $(4; 2, 2, 3, 5^{-})$-star, a $(4; 2, 3, 3, 3)$-star, a $(5; 2, 2, 2, 2, 2)$-star and a $(5; 2, 2, 2, 2, 3)$-star. 
\item If $G$ is a graph with $\delta(G) = 2$, average degree less than $3$ and without $(2, 2, \infty)$-triangles, then $G$ has one of the following configurations: a $(2, 3^{-}, 3^{-})$-path, a $(2, 2, \infty, 2)$-path, a $(2, 2, 4, 3)$-path, a $(4; 2, 2, 2, 6^{-})$-star, a $(2, 4, 3, 2)$-path, a $(2, 4, 3)$-triangle, a $(5; 2, 2, 2, 2, 2)$-star and a $(5; 2, 2, 2, 2, 3)$-star. 
\item If $G$ is a graph with $\delta(G) \geq 2$ and average degree less than $\frac{10}{3}$, then $G$ has one of the following configurations: a $(2, 2, \infty)$-path, a $(2, 3, 6^{-})$-path, a $(3, 3, 3)$-path, a $(2, 4, 3^{-})$-path and a $(2, 9^{-}, 2)$-path. 
\item If $G$ is a graph with $\delta(G) = 3$ and average degree less than $4$, then $G$ contains a $(4^{-}, 3, 7^{-})$-path, or a $(5, 3, 5)$-path or a $(5, 3, 6)$-path. 
\item If $G$ is a triangle-free normal plane map, then it contains one of the following configurations: a $(3, 3, 3)$-path, a $(3, 3, 4)$-path, a $(3, 3, 5, 3)$-path, a $(4, 3, 4)$-path, a $(4, 3, 5)$-path, a $(5, 3, 5)$-path, a $(5, 3, 6)$-path and a $(3, 4, 3)$-path. 
\end{enumerate}
\end{abstract}
\section{Introduction}
All the graphs in this paper are finite, undirected, and simple, unless otherwise stated. We use $V(G)$, $E(G)$, $\Delta(G)$, and $\delta(G)$ to denote the vertex set, the edge set, the maximum degree, and the minimum degree of $G$, respectively. Let $G$ be a plane graph, we use $F(G)$ to denote the face set of $G$. The degree of a vertex $v$ is denoted by $\deg(v)$, and the degree (or face size) of a face $\alpha$, is denoted by $\deg(\alpha)$. A $\kappa$-vertex is a vertex with degree $\kappa$. We denote a $\kappa^{+}$-vertex and $\kappa^{-}$-vertex for a vertex with degree at least $\kappa$ and at most $\kappa$, respectively. Similarly, we can define $\kappa$-faces, $\kappa^{+}$-faces and $\kappa^{-}$-faces. A $(a_{1}, a_{2}, \dots, a_{\kappa})$-path is a path $v_{1}v_{2}\dots v_{\kappa}$ with $\deg(v_{i}) = a_{i}$ for $1 \leq i \leq \kappa$. Similar to the vertices and faces, we can define the $(a_{1}^{-}, a_{2}^{-}, \dots, a_{\kappa}^{-})$-paths. A $(\kappa; a_{1}, a_{2}, \dots, a_{\kappa})$-star in $G$ is a star with center having degree $\kappa$ and all the other vertices with degree $a_{1}, a_{2}, \dots, a_{\kappa}$ in $G$. Let $\omega_{\kappa}$ be the minimum degree-sum of a path on $\kappa$ vertices.

A graph $H$ is light in a family $\mathfrak{G}$ of graphs if at least one member of $\mathfrak{G}$ contains a copy of $H$ and there exists an integer $\lambda(H, \mathfrak{G})$ such that each member $G$ of $\mathfrak{G}$ with a copy of $H$ also has a copy $K$ of $H$ such that $\deg_{G}(v) \leq \lambda(H, \mathfrak{G})$ for all $v \in V(K)$. Note that not every member of $\mathfrak{G}$ contains a copy of $H$ even if $H$ is light in $\mathfrak{G}$. For example, the graph $K_{5}$ is light in the family of graphs $\mathfrak{G} = \{\textrm{planar graphs}\} \cup \{K_{6}\}$. But almost all the results concerning light graph $H$ are subgraphs of each member $G$ in $\mathfrak{G}$. Inspired by this, we defined {\em strongly light graph} in \cite{MR3448602}, a graph $H$ is {\em strongly light} in a family of graphs $\mathfrak{G}$, if there exists an integer $\lambda$, such that every graph $G$ in $\mathfrak{G}$ contains a subgraph $K$ isomorphic to $H$ with $\deg_{G}(v) \leq \lambda$ for all $v \in V(K)$. 

A normal plane map (NPM) is a plane multigraph with minimum vertex degree at least three and minimum face size at least three (multiple edges and loops are allowed). As proved by Steinitz, a graph is polyhedral if and only if it is planar and 3-connected. 

Due to the Euler's formula, every planar graph has a $5^{-}$-vertex, that is, $\omega_{1} \leq 5$. In 1955, Kotzig \cite{MR0074837} showed that every $3$-connected planar graph has an edge of weight at most $13$, and this bound is best possible. Kotzig's result was generalized in various directions since then. Borodin \cite{MR977440} showed that every normal plane map has an edge of weight at most $13$, and this bound is best possible. Jendrol' presented a strong form of Borodin's result in \cite{MR1739913}, stating every normal plane map has a $(3, 10^{-})$-edge, a $(4, 7^{-})$-edge or a $(5, 6^{-})$-edge. 

Ando, Iwasaki and Kaneko \cite{Ando1993} showed that every $3$-connected planar graph has $\omega_{3} \leq 21$, and the bound is best possible. Jendrol' further showed that 
\begin{theorem}[Jendrol' \cite{MR1485387}]
Every $3$-connected planar graph has at least one of the following configurations: a $(10^{-}, 3, 10^{-})$-path, a $(7^{-}, 4, 7^{-})$-path, a $(6^{-}, 5, 6^{-})$-path, a $(3, 4^{-}, 15^{-})$-path, a $(3, 6^{-}, 11^{-})$-path, a $(3, 8^{-}, 5^{-})$-path, a $(3, 10^{-}, 3)$-path, a $(4, 4, 11^{-})$-path, a $(4, 5, 7^{-})$-path and a $(4, 7^{-}, 5^{-})$-path. 
\end{theorem}

The requirement of $3$-connectedness is essential for the finiteness of $\omega_{3}$. Lebesgue \cite{MR0001903} proved that every normal plane map with girth at least five has a $(3, 3, 3)$-path. Borodin \etal \cite{MR3106442} gave a tight description of $3$-paths in normal plane maps.

\begin{theorem}[Borodin, Ivanova, Jensen, Kostochka and Yancey \cite{MR3106442}]
Every normal plane map without two adjacent $3$-vertices lying in two common 3-faces has a $3$-path of the following types: $(3, 4^{-}, 11^{-})$, $(3, 7^{-}, 5^{-})$, $(3, 10^{-}, 4^{-})$, $(3, 15^{-}, 3)$, $(4, 4^{-}, 9^{-})$, $(6^{-}, 4^{-}, 8^{-})$, $(7^{-}, 4^{-}, 7^{-})$, $(6^{-}, 5^{-}, 6^{-})$. 
\end{theorem}

Some of light graphs have been used to prove results on coloring problem, see \cite{MR977440}.  The light subgraphs have been extensively studied, we refer the readers to a recent survey \cite{MR3004475}.

Recently, Borodin and Ivanova \cite{MR3357780} gave a tight description of $3$-paths in triangle-free normal plane maps. 

\begin{theorem}[Borodin and Ivanova \cite{MR3357780}]\mbox{}
\begin{enumerate}
\item Every triangle-free normal plane map has a $(5^{-}, 3, 6^{-})$-path or a $(4^{-}, 3, 7^{-})$-path, which description is tight. 
\item Every triangle-free normal plane map has a $(5^{-}, 3, 6^{-})$-path or a $(3, 4^{-}, 3)$-path, which description is tight. 
\item Every triangle-free normal plane map has a $(3, 5^{-}, 3)$-path or a $(3, 4^{-}, 4^{-})$-path, which description is tight. 
\item Every triangle-free normal plane map has a $(3, 5^{-}, 3)$-path or a $(4^{-}, 3, 4^{-})$-path, which description is tight. \qed
\end{enumerate}
\end{theorem}

Jendrol' and Macekov{\'a} \cite{MR3279266} investigated the 3-paths in plane graphs with minimum degree at least two and some girth restrictions. 
\begin{theorem}[Jendrol' and Macekov{\'a} \cite{MR3279266}]
Every connected plane graph with minimum degree $\delta(G) \geq 2$ and girth at least $g$ has a $3$-path of one of the following types:
\begin{enumerate}
\item $(2, \infty, 2)$, $(2, 2, 6^{-})$, $(2, 3, 5^{-})$, $(2, 4, 4^{-})$ and $(3, 3^{-}, 3)$, if $g \geq 5$;
\item $(2, 2, \infty)$, $(2, 3, 5^{-})$, $(2, 4, 3^{-})$ and $(2, 5, 2)$, if $g = 6$;
\item $(2, 2, 6^{-})$, $(2, 3, 3^{-})$ and $(2, 4, 2)$, if $g = 7$;
\item $(2, 2, 5^{-})$ and $(2, 3, 3^{-})$, if $g \in \{8, 9\}$;
\item $(2, 2, 3^{-})$ and $(2, 3, 2)$, if $g \geq 10$;
\item $(2, 2, 2)$ if $g \geq 16$. \qed
\end{enumerate}
\end{theorem}

Aksenov, Borodin and Ivanova \cite{MR3414174} gave some results on the weight of $3$-paths in plane graphs with minimum degree at least two and girth at least six. 
\begin{theorem}[Aksenov, Borodin and Ivanova \cite{MR3414174}]
Every plane graph $G$ with $\delta(G) \geq 2$ and girth at least six has a $(2, 2, \infty, 2)$-path or $\omega_{3} \leq 9$, and the bound is tight.
\end{theorem}

\begin{theorem}[Aksenov, Borodin and Ivanova \cite{MR3414174}]
Every plane graph $G$ with $\delta(G) \geq 2$ and girth at least seven has $\omega_{3} \leq 9$. 
\end{theorem}

Many of the results on light subgraphs are about planar graphs and proved by the discharging method, and few results are about other graph classes. The average degree of a graph $G$ is the value $\frac{2|E(G)|}{|V(G)|}$, and the maximum average degree $\mad(G)$ of a graph $G$ is the maximum value of the average degree of every subgraph $H$ of $G$, that is, 
\[
\mad(G) = \max_{H \subseteq G}\left\{\frac{2|E(H)|}{|V(H)|}\right\}. 
\] 
The maximum average degree is a measure of the sparseness of graphs, which is used to give global and local structures. Jendrol' \etal \cite{JendrolMacekovaMontassierEtAl2016} gave the following results on $3$-paths in terms of average degree.
\begin{theorem}
Let $G$ be a graph with minimum degree at least two and average degree less than $m$. Then the graph $G$ contains a $3$-path of one of the following types:
\begin{enumerate}[label = (\roman*)]
\item $(2, \infty, 2)$, $(2, 8^{-}, 3^{-})$, $(4^{-}, 3^{-}, 5^{-})$ if $m = \frac{15}{4}$;
\item $(2, \infty, 2)$, $(2, 5^{-}, 3^{-})$, $(3^{-}, 2, 4^{-})$, $(3, 3, 3)$ if $m = \frac{10}{3}$;
\item $(2, 2, \infty)$, $(2, 3^{-}, 4^{-})$, $(2, 5^{-}, 2)$ if $m = 3$;
\item $(2, 2, 13^{-})$, $(2, 3, 3)$, $(2, 4^{-}, 2)$ if $m = \frac{14}{5}$;
\item $(2, 2, \kappa^{-})$, $(2, 3, 2)$ if $m =\frac{3(\kappa+1)}{\kappa+2}$ for $4 \leq \kappa \leq 7$;
\item $(2, 2, 3^{-})$ if $m = \frac{12}{5}$, and 
\item $(2, 2, 2)$ if $m = \frac{9}{4}$. \qed
\end{enumerate}
\end{theorem}
In this paper, we give some light subgraphs in several classes of graphs with the conditions on (maximum) average degree, namely \autoref{MADTHM}--\ref{THMLast}, which refines some results mentioned in the above. 

\section{Light subgraphs}
We first give a lemma on an inequality which is used in the next theorem. 
\begin{lemma}\label{L}
\begin{equation*}
\kappa - (2 + 2 \rho) - 2\kappa\rho \geq 0, 
\end{equation*}
if one of the following holds:
\begin{enumerate}[label = (\arabic*)]
\item\label{a} $\kappa \geq 4$ and $\rho \leq \frac{1}{5}$; 
\item\label{b} $\kappa \geq 5$ and $\rho \leq \frac{1}{4}$; 
\item\label{c} $\kappa \geq 6$ and $\rho \leq \frac{2}{7}$; 
\item\label{d} $\kappa \geq 7$ and $\rho \leq \frac{5}{16}$; 
\item\label{e} $\kappa \geq 8$ and $\rho \leq \frac{1}{3}$. \qed
\end{enumerate}
\end{lemma}

\begin{figure}%
\centering
\subcaptionbox{\label{fig:subfig:a}}{\includegraphics[width = 0.15\textwidth]{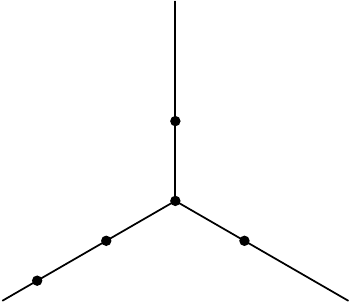}}\hfill~
\subcaptionbox{\label{fig:subfig:b}}{\includegraphics[width = 0.15\textwidth]{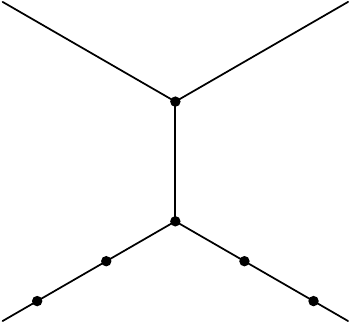}}\hfill~
\subcaptionbox{\label{fig:subfig:c}}{\includegraphics[width = 0.15\textwidth]{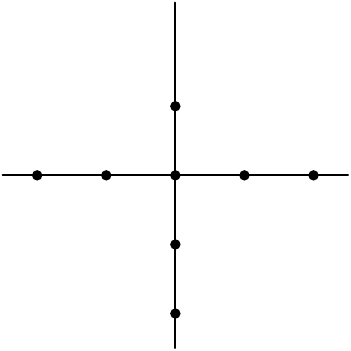}}\hfill~
\subcaptionbox{\label{fig:subfig:d}}{\includegraphics[width = 0.15\textwidth]{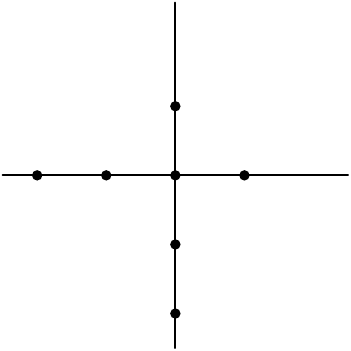}}\hfill~
\subcaptionbox{\label{fig:subfig:e}}{\includegraphics[width = 0.15\textwidth]{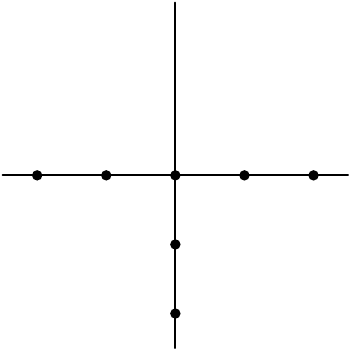}}\hfill~
\subcaptionbox{\label{fig:subfig:f}}{\includegraphics[width = 0.15\textwidth]{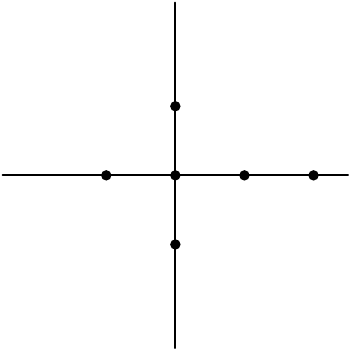}}\\
\subcaptionbox{\label{fig:subfig:g}}{\includegraphics[width = 0.15\textwidth]{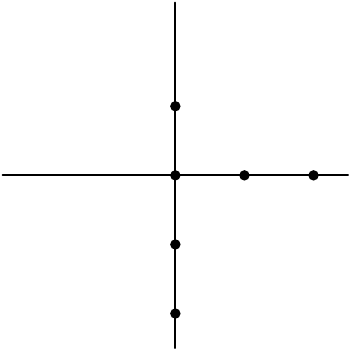}}\hfill~
\subcaptionbox{\label{fig:subfig:h}}{\includegraphics[width = 0.15\textwidth]{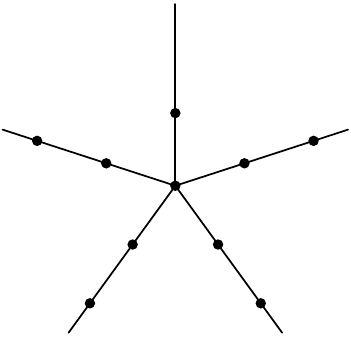}}\hfill~
\subcaptionbox{\label{fig:subfig:i}}{\includegraphics[width = 0.15\textwidth]{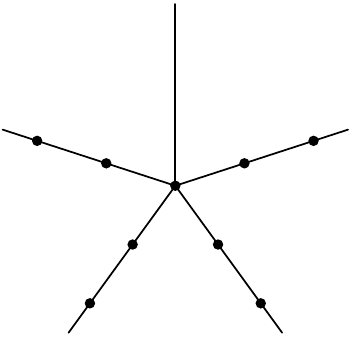}}\hfill~
\subcaptionbox{\label{fig:subfig:j}}{\includegraphics[width = 0.15\textwidth]{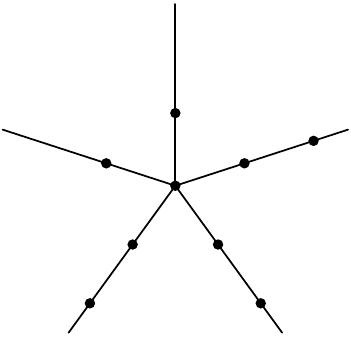}}\hfill~
\subcaptionbox{\label{fig:subfig:k}}{\includegraphics[width = 0.15\textwidth]{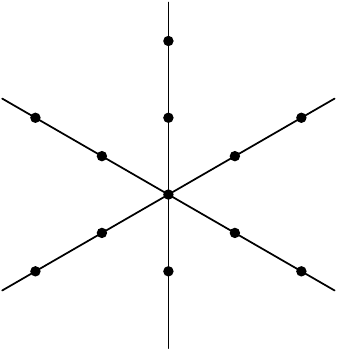}}\hfill~
\subcaptionbox{\label{fig:subfig:l}}{\includegraphics[width = 0.15\textwidth]{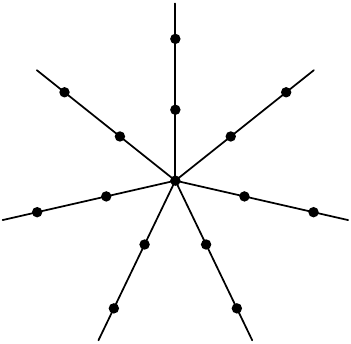}}
\end{figure}

An $\ell$-thread in a graph $G$ is a path of length $\ell+1$ in $G$ whose $\ell$ internal vertices have degree $2$ in the full graph $G$. Some of the configurations in the next theorem are illustrated in the Figs.~\subref{fig:subfig:a}--\subref{fig:subfig:l}. Note that all the edges incident with black dots are represented in the figures.  
\begin{theorem}\label{MADTHM}
Suppose that $G$ is a graph with $\delta(G) = 2$, average degree less than $2 + 2 \rho$ and without $(2, 2, \infty)$-triangles. 
\begin{enumerate}[label = (\arabic*)]
\item\label{aa} If $\rho \leq \frac{1}{5}$, then there exists a $(2, 2, 2)$-path or a configuration in Fig.~\subref{fig:subfig:a}, \subref{fig:subfig:b}. 
\item\label{bb} If $\rho \leq \frac{1}{4}$, then there exists a $(2, 2, 2)$-path, a $(2, 2, 3, 2)$-path, a $(3; 2, 2, 2)$-star, or a configuration in Fig.~\subref{fig:subfig:c}. 
\item\label{cc} If $\rho \leq \frac{2}{7}$, then there exists a $(2, 2, 2)$-path, a $(2, 2, 3)$-path, a $(3; 2, 2, 5^{-})$-star, a configuration in Fig.~\subref{fig:subfig:d}, \subref{fig:subfig:e}, or \subref{fig:subfig:h}. 
\item\label{dd} If $\rho \leq \frac{5}{16}$, then there exists a $(2, 2, 2)$-path, a $(2, 2, 3)$-path, a $(2, 3, 2)$-path, a configuration in Fig.~\subref{fig:subfig:f}, \subref{fig:subfig:g}, \subref{fig:subfig:i}, \subref{fig:subfig:j}, or \subref{fig:subfig:k}. 
\item\label{ee} If $\rho \leq \frac{1}{3}$, then there exists a $(2, 2, 2)$-path, a $(2, 2, 3)$-path, a $(2, 3, 2)$-path, a configuration in Fig.~\subref{fig:subfig:f}, \subref{fig:subfig:g}, \subref{fig:subfig:i}, \subref{fig:subfig:j}, \subref{fig:subfig:k} or \subref{fig:subfig:l}. 
\end{enumerate}
\end{theorem}
\begin{proof}
Suppose that all the configurations do not exist. We set the initial charge of every $\kappa$-vertex with $\kappa - (2 + 2 \rho)$. Thus the sum of all the initial charge of vertices is less than zero. We design appropriate discharging rules to redistribute charge among the vertices, such that the final charge of each vertex is at least zero, and then the sum of the final charge is at least zero, which is a contradiction. 

\begin{enumerate}[label = (R\arabic*)]
\item Each $2$-vertex receives $\rho$ from each endpoint of the maximal thread containing it. 
\item When $\rho \leq \frac{1}{5}$, each $4^{+}$-vertex sends $2\rho$ to each adjacent $3^{+}$-vertex. 
\item When $\rho \leq \frac{2}{7}$, each $6^{+}$-vertex sends $2\rho$ to each adjacent $3^{+}$-vertex. 
\item When $\rho \leq \frac{5}{16}$, each $7^{+}$-vertex sends $2\rho$ to each adjacent $3^{+}$-vertex. 
\item When $\rho \leq \frac{1}{3}$, each $7^{+}$-vertex sends $2\rho$ to each adjacent $3^{+}$-vertex. 
\end{enumerate}

Every $2$-vertex has final charge zero due to (R1). By the absence of $(2, 2, 2)$-paths, every $\kappa$-vertex sends at most $2\kappa$ times $\rho$, thus the final charge is at least $\kappa - (2 + 2 \rho) - 2\kappa\rho$. Note that no vertex is contained in a $(2, 2, \infty)$-triangle. 

\paragraph{Proof of the item \ref{aa}.} By \autoref{L}~\ref{a}, the final charge of every $4^{+}$-vertex is nonnegative. Now, we consider the $3$-vertices. If a $3$-vertex is adjacent to a $4^{+}$-vertex, then its final charge is at least $3 - (2 + 2 \rho) + 2\rho - 4\rho > 0$. Suppose that every $3$-vertex is adjacent to three $3^{-}$-vertices. By the absence of the configurations in Fig.~\subref{fig:subfig:a} and \subref{fig:subfig:b}, a $3$-vertex sends at most three times $\rho$, thus its final charge is at least $3 - (2 + 2 \rho) - 3\rho \geq 0$. 

\paragraph{Proof of the item \ref{bb}.} By \autoref{L}~\ref{b}, the final charge of every $5^{+}$-vertex is nonnegative. Now, we consider the $4^{-}$-vertices. By the absence of $(2, 2, 3, 2)$-paths and $(3; 2, 2, 2)$-stars, a $3$-vertex $w$ sends at most two times $\rho$, thus its final charge is at least $3 - (2 + 2 \rho) - 2\rho \geq 0$. By the absence of the configuration in Fig.~\subref{fig:subfig:c}, a $4$-vertex $w$ sends at most six times $\rho$, thus its final charge is at least $4 - (2 + 2 \rho) - 6\rho \geq 0$. 

\paragraph{Proof of the item \ref{cc}.} By \autoref{L}~\ref{c}, the final charge of every $6^{+}$-vertex is nonnegative. Now, we consider the $5^{-}$-vertices. Note that $(2, 2, 2)$-paths and $(2, 2, 3)$-paths are absent in $G$. If a $3$-vertex $w$ is adjacent to at most one $2$-vertex, then its final charge is at least $3 - (2 + 2 \rho) - \rho > 0$. Suppose that a $3$-vertex $w$ is adjacent to two $2$-vertices. By the absence of $(3; 2, 2, 5^{-})$-stars, $w$ is adjacent to a $6^{+}$-vertex, thus its final charge is at least $3 - (2 + 2 \rho) + 2\rho - 2\rho = 1- 2\rho > 0$. 

By the absence of the configurations in Fig.~\subref{fig:subfig:d} and Fig.~\subref{fig:subfig:e}, a $4$-vertex $w$ sends at most five times $\rho$, thus its final charge is at least $4 - (2 + 2 \rho) - 5\rho \geq 0$. By the absence of the configuration in Fig.~\subref{fig:subfig:h}, a $5$-vertex sends at most eight times $\rho$, thus its final charge is at least $5 - (2 + 2 \rho) - 8\rho > 0$. 

\paragraph{Proof of the item \ref{dd}.} By \autoref{L}~\ref{d}, the final charge of every $7^{+}$-vertex is nonnegative. Now, we consider the $6^{-}$-vertices. By the absence of $(2, 2, 3)$-paths and $(2, 3, 2)$-paths, a $3$-vertex $w$ sends at most one times $\rho$, thus its final charge is at least $3 - (2 + 2 \rho) - \rho > 0$. 

If a $4$-vertex is adjacent to a $7^{+}$-vertex, then its final charge is at least $4 - (2 + 2\rho) + 2\rho - 6\rho > 0$. Suppose that a $4$-vertex $w$ is not adjacent to any $7^{+}$-vertex. By the absence of the configurations in Fig.~\subref{fig:subfig:f} and Fig.~\subref{fig:subfig:g}, $w$ sends at most four times $\rho$, thus its final charge is at least $4 - (2 + 2 \rho) - 4\rho > 0$. 

If a $5$-vertex is adjacent to a $7^{+}$-vertex, then its final charge is at least $5 - (2 + 2\rho) + 2\rho - 8\rho > 0$. Suppose that a $5$-vertex $w$ is not adjacent to any $7^{+}$-vertex. By the absence of the configurations in Fig.~\subref{fig:subfig:i} and Fig.~\subref{fig:subfig:j}, $w$ sends at most seven times $\rho$, thus its final charge is at least $5 - (2 + 2 \rho) - 7\rho > 0$. By the absence of the configuration in Fig.~\subref{fig:subfig:k}, a $6$-vertex sends at most ten times $\rho$, thus its final charge is at least $6 - (2 + 2 \rho) - 10\rho > 0$. 

\paragraph{Proof of the item \ref{ee}.} By \autoref{L}~\ref{e}, the final charge of every $8^{+}$-vertex is nonnegative. Now, we consider the $7^{-}$-vertices. By the absence of $(2, 2, 3)$-paths and $(2, 3, 2)$-paths, a $3$-vertex $w$ sends at most one times $\rho$, thus its final charge is at least $3 - (2 + 2 \rho) - \rho \geq 0$. 

If a $4$-vertex is adjacent to a $8^{+}$-vertex, then its final charge is at least $4 - (2 + 2\rho) + 2\rho - 6\rho \geq 0$. Suppose that a $4$-vertex $w$ is not adjacent to any $8^{+}$-vertex. By the absence of the configurations in Fig.~\subref{fig:subfig:f} and Fig.~\subref{fig:subfig:g}, $w$ sends at most four times $\rho$, thus its final charge is at least $4 - (2 + 2 \rho) - 4\rho \geq 0$. 

If a $5$-vertex is adjacent to a $8^{+}$-vertex, then its final charge is at least $5 - (2 + 2\rho) + 2\rho - 8\rho > 0$. Suppose that a $5$-vertex $w$ is not adjacent to any $8^{+}$-vertex. By the absence of the configurations in Fig.~\subref{fig:subfig:i} and Fig.~\subref{fig:subfig:j}, $w$ sends at most seven times $\rho$, thus its final charge is at least $5 - (2 + 2 \rho) - 7\rho \geq 0$. By the absence of the configuration in Fig.~\subref{fig:subfig:k}, a $6$-vertex sends at most ten times $\rho$, thus its final charge is at least $6 - (2 + 2 \rho) - 10\rho \geq 0$. 

By the absence of the configuration in Fig.~\subref{fig:subfig:l}, a $7$-vertex sends at most thirteen times $\rho$, thus its final charge is at least $7 - (2 + 2 \rho) - 13\rho \geq 0$. 
\end{proof}

Note that $(\girth(G) - 2)\mad(G) < 2\girth(G)$ for every planar graph $G$, so we immediately have the following results. 

\begin{corollary}\label{SparsePlane}
Suppose that $G$ is a planar graph with $\delta(G) = 2$ and girth at least $g$. 
\begin{enumerate}[label = (\arabic*)]
\item If $g \geq 12$, then there exists a $(2, 2, 2)$-path, or a configuration in Fig.~\subref{fig:subfig:a}, \subref{fig:subfig:b}. 
\item If $g \geq 10$, then there exists a $(2, 2, 2)$-path, a $(2, 2, 3, 2)$-path, a $(3; 2, 2, 2)$-star, or a configuration in Fig.~\subref{fig:subfig:c}. 
\item If $g \geq 9$, then there exists a $(2, 2, 2)$-path, a $(2, 2, 3)$-path, a $(3; 2, 2, 5^{-})$-star, a configuration in Fig.~\subref{fig:subfig:d}, \subref{fig:subfig:e}, or \subref{fig:subfig:h}.  
\item If $g \geq 8$, then there exists a $(2, 2, 2)$-path, a $(2, 2, 3)$-path, a $(2, 3, 2)$-path, a configuration in Fig.~\subref{fig:subfig:f}, \subref{fig:subfig:g}, \subref{fig:subfig:i}, \subref{fig:subfig:j}, \subref{fig:subfig:k} or \subref{fig:subfig:l}. \qed
\end{enumerate}
\end{corollary}

Now, we consider the sharpness of \autoref{MADTHM}. For each configuration \subref{fig:subfig:a}--\subref{fig:subfig:l}, we construct a class of graphs containing that configuration but no the others. We may assume that the number of vertices in $R$ is $n$ and the number of edges in $R$ is $m$ in the following. Note that the configuration \subref{fig:subfig:a} contains a $(2, 2, 3, 2)$-path and a $(3; 2, 2, 2)$-star; the configuration \subref{fig:subfig:b} contains a $(2, 2, 3, 2)$-path and a $(3; 2, 2, 5^{-})$-star. 

{\bf $(2, 2, 2)$-path}: Let $G_{1}$ be the graph obtained from a $4$-regular graph by inserting three vertices on each edge. Note that $\mad(G_{1}) = \frac{16}{7} < 2 + 2 \times \frac{1}{5}$ and $G_{1}$ contains a $(2, 2, 2)$-path but no configuration \subref{fig:subfig:a} or \subref{fig:subfig:b}. Let $G_{2}$ be the graph obtained from an $8$-regular graph by inserting three vertices on each edge. Note that $\mad(G_{2}) = \frac{32}{13} < 2 + 2 \times \frac{1}{4}$ and $G_{2}$ contains a $(2, 2, 2)$-path. 

{\bf $(2, 2, 3, 2)$-path}: Let $R$ be a $3$-regular graph with a $3$-edge-coloring (\eg the complete graph on four vertices). Let $G$ be the graph obtained from $R$ by inserting two vertices on each edge in the class $M_{1}$ and one vertex on each edge in the class $M_{2}$. Note that $|V(G)| = n + m$ and $|E(G)| = 2m$. Hence, the graph $G$ has average degree $2 + \frac{2}{5}$ and it contains a $(2, 2, 3, 2)$-path. This graph also shows that the upper bound on the average degree in \autoref{MADTHM}~\ref{aa} is best possible.

{\bf $(3; 2, 2, 2)$-star}: Let $G$ be the graph obtained from a $3$-regular graph $R$ by inserting one vertex on each edge. Note that $|V(G)| = n + m$ and $|E(G)| = 2m$. Hence, the graph $G$ has average degree $2 + \frac{2}{5}$ and it contains a $(3; 2, 2, 2)$-star.

{\bf $(2, 2, 3)$-path}: Let $R$ be a $3$-regular graph with a perfect matching $M$ (\eg the complete graph on four vertices). Let $G$ be the graph obtained from $R$ by inserting two vertices on each edge in $M$. Note that $|V(G)| = n + \frac{2}{3}m$ and $|E(G)| = \frac{5}{3}m$. Hence, the graph $G$ has average degree $2 + \frac{1}{2}$ and it contains a $(2, 2, 3)$-path. This graph also shows that the upper bound on the average degree in \autoref{MADTHM}~\ref{bb} is best possible. 

{\bf $(3; 2, 2, 5^{-})$-star/$(2, 3, 2)$-path}: Let $R$ be a $3$-regular graph with a perfect matching $M$ (\eg the complete graph on four vertices). Let $G$ be the graph obtained from $R$ by inserting one vertex on each edge not in $M$. Note that $|V(G)| = n + \frac{2}{3}m$ and $|E(G)| = \frac{5}{3}m$. Hence, the graph $G$ has average degree $2 + \frac{1}{2}$ and it contains a $(3; 2, 2, 5^{-})$-star/$(2, 3, 2)$-path. 

{\bf Configuration \subref{fig:subfig:a}}: Let $R$ be a $3$-regular graph with a perfect matching $M$ (\eg the complete graph on four vertices). Let $G$ be the graph obtained from $R$ by inserting two vertices on each edge in $M$ and one vertex on each edge not in $M$. Note that $|V(G)| = n + \frac{4}{3}m$ and $|E(G)| = \frac{7}{3}m$. Hence, the graph $G$ has average degree $2 + \frac{1}{3}$ and it contains a configuration \subref{fig:subfig:a}.

{\bf Configuration \subref{fig:subfig:b}}: Let $R$ be a $3$-regular graph with a perfect matching $M$ (\eg the complete graph on four vertices). Let $G$ be the graph obtained from $R$ by inserting two vertices on each edge not in $M$. Note that $|V(G)| = n + \frac{4}{3}m$ and $|E(G)| = \frac{7}{3}m$. Hence, the graph $G$ has average degree $2 + \frac{1}{3}$ and it contains a configuration \subref{fig:subfig:b}.

{\bf Configuration \subref{fig:subfig:c}}: Let $R$ be a $4$-regular graph with a perfect matching $M$ (\eg $C_{2k} \times C_{\ell}$). Let $G$ be the graph obtained from $R$ by inserting one vertex on each edge in $M$ and two vertices on each edge not in $M$. Note that $|V(G)| = n + \frac{7}{4}m$ and $|E(G)| = \frac{11}{4}m$. Hence, the graph $G$ has average degree $2 + \frac{4}{9}$ and it contains a configuration \subref{fig:subfig:c}. 

{\bf Configuration \subref{fig:subfig:d}}: Let $R$ be a $4$-regular graph with a $2$-factor $F$ (\eg $C_{2k} \times C_{\ell}$). Let $G$ be the graph obtained from $R$ by inserting two vertices on each edge in $F$ and one vertex on each edge not in $F$. Note that $|V(G)| = n + \frac{3}{2}m$ and $|E(G)| = \frac{5}{2}m$. Hence, the graph $G$ has average degree $2 + \frac{1}{2}$ and it contains a configuration \subref{fig:subfig:d}. 

{\bf Configuration \subref{fig:subfig:e}}: Let $R$ be a $4$-regular graph with a perfect matching $M$ (\eg $C_{2k} \times C_{\ell}$). Let $G$ be the graph obtained from $R$ by inserting two vertices on each edge not in $M$. Note that $|V(G)| = n + \frac{3}{2}m$ and $|E(G)| = \frac{5}{2}m$. Hence, the graph $G$ has average degree $2 + \frac{1}{2}$ and it contains a configuration \subref{fig:subfig:e}. 

{\bf Configuration \subref{fig:subfig:f}}: Let $R$ be a $4$-regular graph with a perfect matching $M$ (\eg $C_{2k} \times C_{\ell}$). Let $G$ be the graph obtained from $R$ by inserting two vertices on each edge in $M$ and one vertex on each edge not in $M$. Note that $|V(G)| = n + \frac{5}{4}m$ and $|E(G)| = \frac{9}{4}m$. Hence, the graph $G$ has average degree $2 + \frac{4}{7}$ and it contains a configuration \subref{fig:subfig:f}. This graph also shows that the upper bound on the average degree in \autoref{MADTHM}~\ref{cc} is best possible. 

{\bf Configuration \subref{fig:subfig:g}}: Let $R$ be a $4$-regular graph with a $2$-factor $F$ in which each component is an even cycle (\eg $C_{2k} \times C_{\ell}$). Let $G$ be the graph obtained from $R$ by inserting two vertices on each edge not in $F$ and one vertex on each edge of a perfect matching of $F$. Note that $|V(G)| = n + \frac{5}{4}m$ and $|E(G)| = \frac{9}{4}m$. Hence, the graph $G$ has average degree $2 + \frac{4}{7}$ and it contains a configuration \subref{fig:subfig:g}. 

{\bf Configuration \subref{fig:subfig:h}}: Let $R$ be a $5$-regular graph with a perfect matching $M$ (\eg the complete graph on six vertices). Let $G$ be the graph obtained from $R$ by inserting one vertex on each edge in $M$ and two vertices on each edge not in $M$. Note that $|V(G)| = n + \frac{9}{5}m$ and $|E(G)| = \frac{14}{5}m$. Hence, the graph $G$ has average degree $2 + \frac{6}{11}$ and it contains a configuration \subref{fig:subfig:h}. 

{\bf Configuration \subref{fig:subfig:i}}: Let $R$ be a $5$-regular graph with a perfect matching $M$ (\eg the complete graph on six vertices). Let $G$ be the graph obtained from $R$ by inserting two vertices on each edge not in $M$. Note that $|V(G)| = n + \frac{8}{5}m$ and $|E(G)| = \frac{13}{5}m$. Hence, the graph $G$ has average degree $2 + \frac{3}{5}$ and it contains a configuration \subref{fig:subfig:i}. 

{\bf Configuration \subref{fig:subfig:j}}: Let $R$ be a $5$-regular graph with a $2$-factor $F$ (\eg the complete graph on six vertices). Let $G$ be the graph obtained from $R$ by inserting one vertex on each edge in $F$ and two vertices on each edge not in $F$. Note that $|V(G)| = n + \frac{8}{5}m$ and $|E(G)| = \frac{13}{5}m$. Hence, the graph $G$ has average degree $2 + \frac{3}{5}$ and it contains a configuration \subref{fig:subfig:j}. 

{\bf Configuration \subref{fig:subfig:k}}: Let $R$ be a $6$-regular graph with a perfect matching $M$. Let $G$ be the graph obtained from $R$ by inserting one vertex on each edge in $M$ and two vertices on each edge not in $M$. Note that $|V(G)| = n + \frac{11}{6}m$ and $|E(G)| = \frac{17}{6}m$. Hence, the graph $G$ has average degree $2 + \frac{8}{13}$ and it contains a configuration \subref{fig:subfig:k}. 

{\bf Configuration \subref{fig:subfig:l}}: Let $R$ be a $7$-regular graph (\eg the complete graph on eight vertices). Let $G$ be the graph obtained from $R$ by inserting two vertices on each edge. Note that $|V(G)| = n + 2m$ and $|E(G)| = 3m$. Hence, the graph $G$ has average degree $2 + \frac{5}{8}$ and it contains a configuration \subref{fig:subfig:l}. This graph also shows that the upper bound on the average degree in \autoref{MADTHM}~\ref{dd} is best possible. 
\begin{remark}
By the above discussions, no configuration can be omitted, that is, no one can be deleted from the set of unavoidable configurations. We also see that the upper bounds on the average degree in the first four items of \autoref{MADTHM} cannot be improved. Actually, the upper bound on the average degree in \autoref{MADTHM}~\ref{ee} is also best possible. For example, let $G$ be a graph obtained from an $8$-regular graph by inserting two vertices on each edge; note that the average degree of $G$ is $2 + \frac{2}{3}$. 
\end{remark}

\begin{theorem}\label{MAD14/5}
If $G$ is a graph with $\delta(G) = 2$, average degree less than $\frac{14}{5}$ and without $(2, 2, \infty)$-triangles, then $G$ has one of the following configurations: a $(2, 2, 13^{-}, 2)$-path, a $(2, 3^{-}, 3^{-})$-path and a $(4; 2, 2, 2, 3^{-})$-star. 
\end{theorem}
\begin{proof}
Suppose that $G$ is a counterexample to the theorem. Without loss of generality, we may assume that $G$ is connected. We use the discharging method to get a contradiction. Initially, we set the charge of a $\kappa$-vertex with $\kappa - \frac{14}{5}$. We design appropriate discharging rules to redistribute charge among the vertices, such that the final charge of each vertex is at least zero, and then the sum of the final charge is at least zero, which is a contradiction.
\begin{enumerate}[label = (R\arabic*)]
\item Each $2$-vertex receives $\frac{2}{5}$ from each endpoint of the maximal thread containing it. 
\item Each $4^{+}$-vertex sends $\frac{1}{10}$ to each adjacent $3$-vertex.
\end{enumerate}

Each $2$-vertex has final charge $2 - \frac{14}{5} + 2 \times \frac{2}{5} = 0$. By the absence of $(2, 3^{-}, 3^{-})$-paths, each $3$-vertex is adjacent to at most one $2$-vertex, thus the final charge is at least $3 - \frac{14}{5} + 2 \times \frac{1}{10} - \frac{2}{5} = 0$.  

Let $w$ be an arbitrary vertex with $4 \leq \deg(w) \leq 13$. Suppose that it is an endpoint of a $(2, 2, \infty)$-path. By the absence of $(2, 2, 13^{-}, 2)$-paths, $w$ is adjacent to exactly one $2$-vertex, thus its final charge is at least $\deg(w) - \frac{14}{5} - 2 \times \frac{2}{5} - (\deg(w) - 1) \times \frac{1}{10} > 0$. So we may assume that $w$ is not an endpoint of a $(2, 2, \infty)$-path. If $w$ is a $5^{+}$-vertex, then its final charge is at least $\deg(w) - \frac{14}{5} - \deg(w) \times \frac{2}{5} > 0$. If $w$ is a $4$-vertex and it is adjacent to at most two $2$-vertices, then its final charge is at least $4 - \frac{14}{5} - 2 \times \frac{2}{5} - 2 \times \frac{1}{10} > 0$. If $w$ is a $4$-vertex and it is adjacent to at least three $2$-vertices, then it is adjacent to a $4^{+}$-vertex, and then the final charge is at least $4 - \frac{14}{5} - 3 \times \frac{2}{5} = 0$. 

If $w$ is an arbitrary $14^{+}$-vertex, then its final charge is at least $\deg(w) - \frac{14}{5} - 2\deg(w) \times \frac{2}{5} \geq 0$.
\end{proof}

\begin{theorem}\label{Girth7}
If $G$ is a plane graph with $\delta(G) \geq 2$ and face size at least $7$, then $G$ has a $(2, 2, 5^{-})$-path, or a $(2, 5^{-}, 2)$-path or a $(3, 3, 2, 3)$-path. 
\end{theorem}
\begin{proof}
Suppose that $G$ is a counterexample to the theorem. Without loss of generality, we may assume that $G$ is connected. We still use the discharging method to complete the proof. Initially, we assign the charge $\deg(w) - 4$ to each vertex $w$ and $\deg(\alpha) - 4$ to each face $\alpha$. Thus, the sum of the initial charge of vertices and faces is $-8$. We then redistribute the charge among vertices and faces, so that the final charge of each element is nonnegative, which yields the sum of final charge is nonnegative. 
\begin{enumerate}[label = (R\arabic*)]
\item Each $2$-vertex receives $1$ from each incident face. 
\item Each $3$-vertex receives $\frac{1}{3}$ from each incident face.
\item Each $6^{+}$-vertex sends $\frac{1}{3}$ to each incident face.
\end{enumerate}

Each $2$-vertex has final charge $2 - 4 + 2 \times 1 = 0$. Each $3$-vertex has final charge $3 - 4 + 3 \times \frac{1}{3} = 0$. Each $4$-vertex has final charge as the initial charge zero. Each $5$-vertex is not involved in the discharging rules, thus the final charge is $5 - 4 = 1$. Each $6^{+}$-vertex $w$ has final charge $\deg(w) - 4 - \deg(w) \times \frac{1}{3} \geq 0$. Now, we consider the faces.

{\bf Let $\alpha = w_{0}w_{1}\dots w_{6}$ be a $7$-face.}  Suppose that there are two adjacent $2$-vertices on the boundary, say $w_{1}$ and $w_{2}$. By the absence of $(2, 2, 5^{-})$-paths, $w_{0}$ and $w_{3}$ are $6^{+}$-vertices. If $\alpha$ is incident with at most three $2$-vertices, then the final charge is at least $7 - 4 + 2 \times \frac{1}{3} - 3 \times 1 - 2 \times \frac{1}{3} = 0$. So we may assume that $\alpha$ is incident with four $2$-vertices. Note that exactly two of $w_{4}, w_{5}$ and $w_{6}$ are $2$-vertices. By the absence of $(2, 2, 5^{-})$-paths and $(2, 5^{-}, 2)$-paths, we have that one of $w_{4}, w_{5}$ and $w_{6}$ is a $6^{+}$-vertex. Hence, the final charge of $\alpha$ is at least $7 - 4 + 3 \times \frac{1}{3} - 4 \times 1 = 0$. 

Next, suppose that no two adjacent $2$-vertices are on the boundary of $\alpha$. Thus, $\alpha$ is incident with at most three $2$-vertices. If there are three $2$-vertices on the boundary, then there is a $(2, 6^{+}, 2, 6^{+}, 2)$-walk on the boundary, and then the final charge of $\alpha$ is at least $7 - 4 + 2 \times \frac{1}{3} - 3 \times 1 - 2 \times \frac{1}{3} = 0$. If there are at most one $2$-vertex, then the final charge is at least $7 - 4 - 1 - 6 \times \frac{1}{3} = 0$. And if there are exactly two $2$-vertices and at least one $6^{+}$-vertex, then the final charge is at least $7 - 4 + \frac{1}{3} - 2 \times 1 - 4 \times \frac{1}{3} = 0$. So it suffices to check the subcase: there are exactly two $2$-vertices and no $6^{+}$-vertices. By the absence of $(2, 5^{-}, 2)$-paths, we may assume that $w_{0}$ and $w_{3}$ are $2$-vertices. Furthermore, we have that there are at most three $3$-vertices, for otherwise there is a $(3, 2, 3, 3)$-path on the boundary. Hence, the final charge is at least $7 - 4 - 2 \times 1 - 3 \times \frac{1}{3} = 0$. 

{\bf Let $\alpha = w_{0}w_{1}\dots w_{7}$ be an $8$-face.}  If $\alpha$ is incident with five $2$-vertices, then the other three are all $6^{+}$-vertices due to the absence of $(2, 5^{-}, 2)$-paths, and then the final charge is at least $8 - 4 + 3 \times \frac{1}{3} - 5 \times 1 = 0$. So we may assume that $\alpha$ is incident with at most four $2$-vertices. 

Suppose that there are two adjacent $2$-vertices on the boundary, say $w_{1}$ and $w_{2}$. By the absence of $(2, 2, 5^{-})$-paths, $w_{0}$ and $w_{3}$ are $6^{+}$-vertices. Hence, the final charge is at least $8 - 4 + 2 \times \frac{1}{3} - 4 \times 1 - 2 \times \frac{1}{3} = 0$. Next, suppose that no adjacent $2$-vertices are on the boundary. 

If $\alpha$ is incident with exactly four $2$-vertices, then $\alpha$ is a $(2, 6^{+}, 2, 6^{+}, 2, 6^{+}, 2, 6^{+})$-face, and then its final charge is at least $8 - 4 + 4 \times \frac{1}{3} - 4 \times 1 > 0$. If $\alpha$ is incident with exactly three $2$-vertices, then two $2$-vertices has distance two on the boundary, and then $\alpha$ is incident with a $6^{+}$-vertex, which implies the final charge is at least $8 - 4 + \frac{1}{3} - 3 \times 1 - 4 \times \frac{1}{3} = 0$. If $\alpha$ is incident with at most two $2$-vertices, then the final charge is at least $8 - 4 - 2 \times 1 - 6 \times \frac{1}{3} = 0$. 

{\bf Let $\alpha$ be a $9^{+}$-face.}  By the absence of $(2, 2, 5^{-})$-paths and $(2, 5^{-}, 2)$-paths, the face $\alpha$ sends at most $1 + 1 - \frac{1}{3} = 1 + \frac{1}{3} + \frac{1}{3} = \frac{5}{3}$ in total to any three consecutive vertices on the boundary. That is, $\alpha$ averagely sends $\frac{5}{9}$ to each incident vertex, and thus the final charge is at least $\deg(\alpha) - 4 - \deg(\alpha) \times \frac{5}{9} \geq 0$. 
\end{proof}

Very recently, Jendrol' \etal \cite{MR3431391} showed that if we replace "$(3, 3, 2, 3)$-path" with "$(2, 3, 3)$-path" in \autoref{Girth7}, then "$(2, 2, 5^{-})$-path" can be replaced with "$(2, 2, 3^{-})$-path". 

\begin{corollary}[Aksenov, Borodin and Ivanova \cite{MR3414174}]\label{OMEGA3}
If $G$ is a plane graph with $\delta(G) \geq 2$ and girth at least $7$, then $\omega_{3} \leq 9$. 
\end{corollary}
\begin{remark}
Combining \autoref{MAD14/5} and \autoref{Girth7}, the upper bound for $\omega_{3}$ in \autoref{OMEGA3} could be achieved on a $(2, 2, 5, 2)$-path. In \cite{MR3414174}, Aksenov, Borodin and Ivanova showed the sharpness of upper bound on $\omega_{3}$ by inserting two vertices on each edge of the icosahedron. 
\end{remark}

\begin{figure}
\ContinuedFloat
\centering
\subcaptionbox{\label{fig:subfig:P22INFTY}}{\includegraphics{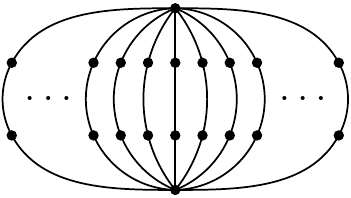}}
\end{figure}
Next, we consider the class of graphs with average degree less than $3$, which includes all the plane graphs with face size at least $6$. 
\begin{theorem}\label{MAD3}
If $G$ is a graph with $\delta(G) = 2$, average degree less than $3$ and without $(2, 2, \infty)$-triangles, then $G$ has one of the following configurations: a $(2, 3^{-}, 3^{-})$-path, a $(2, 2, \infty, 2)$-path, a $(2, 2, 4, 3)$-path, a $(4; 2, 2, 2, 6^{-})$-star, a $(4; 2, 2, 3, 5^{-})$-star, a $(4; 2, 3, 3, 3)$-star, a $(5; 2, 2, 2, 2, 2)$-star and a $(5; 2, 2, 2, 2, 3)$-star. 
\end{theorem}
\begin{proof}
Suppose that all the configurations do not exist. We set the initial charge of every $\kappa$-vertex with $\kappa - 3$. Thus the sum of all the initial charge of vertices is less than zero. We design appropriate discharging rules to redistribute charge among the vertices, such that the final charge of each vertex is at least zero, and then the sum of the final charge is at least zero, which is a contradiction. 

\begin{enumerate}[label = (R\arabic*)]
\item Each $2$-vertex receives $\frac{1}{2}$ from each endpoint of the maximal thread containing it. 
\item Each $4$-vertex sends $\frac{1}{4}$ to each adjacent $3$-vertex that is adjacent to a $2$-vertex. 
\item Each $5$-vertex sends $\frac{1}{4}$ to each adjacent $3$-vertex. 
\item Each $6$-vertex sends $\frac{1}{4}$ to each adjacent $3^{+}$-vertex. 
\item Each $7^{+}$-vertex sends $\frac{1}{2}$ to each adjacent $3^{+}$-vertex.
\end{enumerate}

{\bf Let $w$ be a $3$-vertex.} If $w$ is adjacent to three $3^{+}$-vertices, then its final charge is at least zero. By the absence of $(2, 3^{-}, 3^{-})$-paths, $w$ is not an endpoint of a $(2, 2, 3)$-path, and we may also assume that $w$ is adjacent to exactly one $2$-vertex and two $4^{+}$-vertices, thus its final charge is at least $3 - 3 + 2 \times \frac{1}{4} - \frac{1}{2} = 0$. 

{\bf Let $w$ be a $4$-vertex.} If $w$ is an endpoint of a $(2, 2, \infty)$-path, then it sends two times $\frac{1}{2}$, and then its final charge is $4 - 3 - 2 \times \frac{1}{2} = 0$. So we may assume that $w$ is not an endpoint of a $(2, 2, \infty)$-path. By the absence of $(4; 2, 2, 2, 6^{-})$-stars, if $w$ is adjacent to three $2$-vertices, then it is adjacent to a $7^{+}$-vertex, and then its final charge is $4 - 3 + \frac{1}{2} - 3 \times \frac{1}{2} = 0$.  If $w$ is adjacent to exactly two $2$-vertices, then its final charge is at least $4 - 3 - 2\times \frac{1}{2} = 0$ or $4 - 3 + \frac{1}{4} - 2 \times \frac{1}{2} - \frac{1}{4} = 0$. By the absence of $(4; 2, 3, 3, 3)$-stars, the final charge is at least $4 - 3 - \frac{1}{2} - 2 \times \frac{1}{4} = 0$ if $w$ is adjacent to exactly one $2$-vertex and at most two $3$-vertices, or $4 - 3 - 4 \times \frac{1}{4} = 0$ if $w$ is adjacent to four $3^{+}$-vertices. 

{\bf Let $w$ be a $5$-vertex.} By the absence of $(2, 2, \infty, 2)$-paths, if $w$ is an endpoint of a $(2, 2, \infty)$-path, then it sends two times $\frac{1}{2}$, at most four times $\frac{1}{4}$, and then its final charge is at least $5 - 3 - 2 \times \frac{1}{2} - 4 \times \frac{1}{4} = 0$. 

Suppose that $w$ is not an endpoint of a $(2, 2, \infty)$-path. If $w$ is adjacent to at most three $2$-vertices, then its final charge is at least $5 - 3 - 3 \times \frac{1}{2} - 2 \times \frac{1}{4} = 0$. So we may assume that $w$ is adjacent to at least four $2$-vertices. By the absence of $(5; 2, 2, 2, 2, 2)$-stars and $(5; 2, 2, 2, 2, 3)$-stars, it is adjacent to a $4^{+}$-vertex, thus its final charge is at least $5 - 3 - 4 \times \frac{1}{2} = 0$. 

{\bf Let $w$ be a $6$-vertex.} By the absence of $(2, 2, \infty, 2)$-paths, if $w$ is an endpoint of a $(2, 2, \infty)$-path, then it sends two times $\frac{1}{2}$, at most five times $\frac{1}{4}$, and then its final charge is at least $6 - 3 - 2 \times \frac{1}{2} - 5 \times \frac{1}{4} > 0$. If $w$ is not an endpoint of a $(2, 2, \infty)$-path, then it sends at most $\frac{1}{2}$ to each adjacent vertex, and then its final charge is at least $6 - 3 - 6 \times \frac{1}{2} = 0$. 

{\bf Let $w$ be a $7^{+}$-vertex.} By the absence of $(2, 2, 2)$-paths and $(2, 2, \infty, 2)$-paths, $w$ sends at most $\deg(w) + 1$ times $\frac{1}{2}$, thus its final charge is at least $\deg(w) - 3 - (\deg(w) + 1) \times \frac{1}{2} \geq 0$.
\end{proof}

\begin{corollary}[Jendrol' and Macekov{\'a} \cite{MR3279266}]
If $G$ is a planar graph with $\delta(G) = 2$ and girth at least $6$, then $G$ has one of the following configurations: a $(2, 2, \infty)$-path, a $(2, 3, 5^{-})$-path, a $(2, 4, 3^{-})$-path and a $(2, 5, 2)$-path. 
\end{corollary}

\begin{corollary}[Aksenov, Borodin and Ivanova \cite{MR3414174}]
If $G$ is a planar graph with $\delta(G) = 2$ and girth at least $6$, then $G$ has a $(2, 2, \infty, 2)$-path or $\omega_{3} \leq 9$. 
\end{corollary}

Note that we can replace "a $(4; 2, 2, 3, 5^{-})$-star, a $(4; 2, 3, 3, 3)$-star" with "a $(2, 4, 3, 2)$-path, a $(2, 4, 3)$-triangle" in \autoref{MAD3}. It suffices to check the final charge of every $4$-vertex. So we immediately have the following result. 

\begin{theorem}
If $G$ is a graph with $\delta(G) = 2$, average degree less than $3$ and without $(2, 2, \infty)$-triangles, then $G$ has one of the following configurations: a $(2, 3^{-}, 3^{-})$-path, a $(2, 2, \infty, 2)$-path, a $(2, 2, 4, 3)$-path, a $(4; 2, 2, 2, 6^{-})$-star, a $(2, 4, 3, 2)$-path, a $(2, 4, 3)$-triangle, a $(5; 2, 2, 2, 2, 2)$-star and a $(5; 2, 2, 2, 2, 3)$-star. 
\end{theorem}

Next, we consider the class of graphs with average degree less than $\frac{10}{3}$, which includes all the plane graphs with face size at least $5$. 
\begin{theorem}\label{MAD10/3}
If $G$ is a graph with $\delta(G) \geq 2$ and average degree less than $\frac{10}{3}$, then $G$ has one of the following configurations: a $(2, 2, \infty)$-path, a $(2, 3, 6^{-})$-path, a $(3, 3, 3)$-path, a $(2, 4, 3^{-})$-path and a $(2, 9^{-}, 2)$-path. 
\end{theorem}
\begin{proof}
Suppose that all the configurations do not exist. We set the initial charge of every vertex $w$ with $\deg(w) - \frac{10}{3}$. Thus the sum of all the initial charge of vertices is less than zero. We design appropriate discharging rules to redistribute charge among the vertices, such that the final charge of each vertex is at least zero, and then the sum of the final charge is at least zero, which is a contradiction. 

A $3$-vertex is called a $3_{1}$-vertex (or $3_{0}$-vertex) if it is adjacent to exactly one (or zero) $2$-vertex. 
\begin{enumerate}[label = (R\arabic*)]
\item Each $2$-vertex receives $\frac{2}{3}$ from each endpoint of the maximal thread containing it. 
\item Each $3_{0}$-vertex receives $\frac{1}{6}$ from each adjacent $4^{+}$-vertex.
\item Each $3_{1}$-vertex receives $\frac{1}{2}$ from each adjacent $7^{+}$-vertex. 
\end{enumerate}

If $w$ is a $10^{+}$-vertex, then its final charge is at least $\deg(w) - \frac{10}{3} - \deg(w) \times \frac{2}{3} \geq 0$. It suffices to consider the $9^{-}$-vertices. Note that each $9^{-}$-vertex is adjacent to at most one $2$-vertex due to the absence of $(2, 9^{-}, 2)$-paths. 

{\bf Let $w$ be a $3$-vertex.} By the absence of $(2, 3, 2)$-paths, $w$ is a $3_{0}$-vertex or $3_{1}$-vertex. By the absence of $(3, 3, 3)$-paths, if $w$ is a $3_{0}$-vertex, then $w$ is adjacent to at least two $4^{+}$-vertices, and then its final charge is at least $3 - \frac{10}{3} + 2 \times \frac{1}{6} = 0$. By the absence of $(2, 3, 6^{-})$-paths, if $w$ is a $3_{1}$-vertex, then $w$ is adjacent to two $7^{+}$-vertices, and then its final charge is $3 - \frac{10}{3} + 2 \times \frac{1}{2} - \frac{2}{3} = 0$. 

{\bf Let $w$ be a $4$-vertex.} By the absence of $(2, 4, 3^{-})$-paths, if $w$ is adjacent to a $2$-vertex, then it is adjacent to three $4^{+}$-vertices, and its final charge is $4 - \frac{10}{3} - \frac{2}{3} = 0$. By the absence of $(2, 3, 4)$-paths, if $w$ is adjacent to four $3^{+}$-vertices, then all the adjacent $3$-vertices are $3_{0}$-vertices, and its final charge at least $4 - \frac{10}{3} - 4 \times \frac{1}{6} = 0$. 

{\bf Let $w$ be a $5$-vertex.} By the absence of $(2, 3, 6^{-})$-paths, $w$ is not adjacent to any $3_{1}$-vertex, so we have the final charge at least $5 - \frac{10}{3} - \frac{2}{3} - 4 \times \frac{1}{6} > 0$. 

{\bf Let $w$ be a $6$-vertex.} By the absence of $(2, 3, 6^{-})$-paths, $w$ is not adjacent to any $3_{1}$-vertex, so we have the final charge at least $6 - \frac{10}{3} - \frac{2}{3} - 5 \times \frac{1}{6} > 0$. 

If $w$ is a $7$-vertex, then its final charge is at least $7 - \frac{10}{3} - \frac{2}{3} - 6 \times \frac{1}{2} = 0$. If $w$ is an $8$-vertex, then its final charge is at least $8 - \frac{10}{3} - \frac{2}{3} - 7 \times \frac{1}{2} > 0$. If $w$ is a $9$-vertex, then its final charge is at least $9 - \frac{10}{3} - \frac{2}{3} - 8 \times \frac{1}{2} > 0$. 
\end{proof}
\begin{corollary}[Lebesgue \cite{MR0001903}]
Every normal plane map with girth at least $5$ has a $(3, 3, 3)$-path. \qed
\end{corollary}
Now, we consider the sharpness of \autoref{MAD10/3}. The graph in Fig.~\subref{fig:subfig:P22INFTY} contains only the configuration --- $(2, 2, \infty)$-path. The complete bipartite graph $K_{2, \kappa}$ with $\kappa \leq 9$ contains only the configuration --- $(2, \kappa, 2)$-path. The graph in Fig.~\subref{fig:subfig:P233} contains only the configuration --- $(2, 3, 3)$-path. The graph in Fig.~\subref{fig:subfig:P234} contains only the configuration --- $(2, 3, 4)$-path. The graph in Fig.~\subref{fig:subfig:P235} contains only the configuration --- $(2, 3, 5)$-path. The graph in Fig.~\subref{fig:subfig:P236} contains only the configuration --- $(2, 3, 6)$-path. The graph in Fig.~\subref{fig:subfig:P243} contains only the configuration --- $(2, 4, 3)$-path. Each $3$-regular graph only contains the configuration --- $(3, 3, 3)$-path. The upper bound on the average degree is also sharp because of the complete bipartite graph $K_{2, 10}$.

\begin{figure}%
\ContinuedFloat
\centering
\subcaptionbox{\label{fig:subfig:P233}}{\includegraphics[width = 0.15\textwidth]{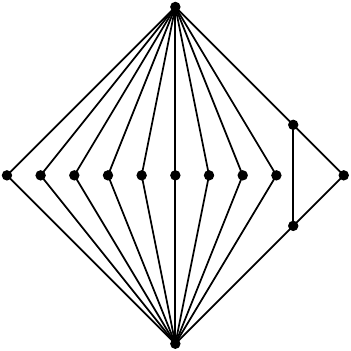}}\hfill~
\subcaptionbox{\label{fig:subfig:P234}}{\includegraphics[width = 0.15\textwidth]{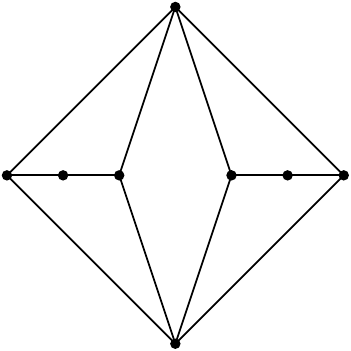}}\hfill~
\subcaptionbox{\label{fig:subfig:P235}}{\includegraphics[width = 0.15\textwidth]{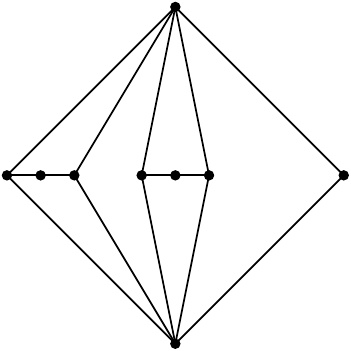}}\hfill~
\subcaptionbox{\label{fig:subfig:P236}}{\includegraphics[width = 0.15\textwidth]{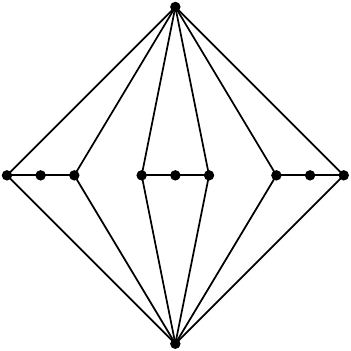}}\hfill~
\subcaptionbox{\label{fig:subfig:P243}}{\includegraphics[width = 0.15\textwidth]{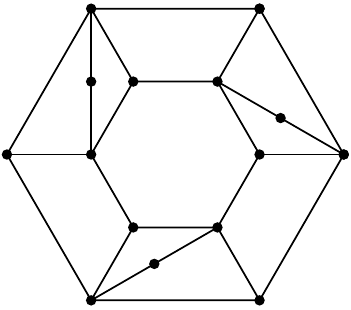}}
\end{figure}

\begin{theorem}
If $G$ is a graph with $\delta(G) = 3$ and average degree less than $4$, then $G$ contains a $(4^{-}, 3, 7^{-})$-path, or a $(5, 3, 5)$-path or a $(5, 3, 6)$-path. 
\end{theorem}
\begin{proof}
Suppose that all the configurations do not exist. We set the initial charge of every $\kappa$-vertex with $\kappa - 4$. Thus the sum of all the initial charge of vertices is less than zero. We design an appropriate discharging rule to redistribute charge among the vertices, such that the final charge of each vertex is at least zero, and then the sum of the final charge is at least zero, which is a contradiction. 

\begin{itemize}
\item[(R)] Each $4^{+}$-vertex $w$ sends $\frac{\deg(w) - 4}{\deg(w)}$ to each adjacent vertex. 
\end{itemize}

It is obvious that each $4^{+}$-vertex has the final charge at least zero. Suppose that $w$ is a $3$-vertex which is adjacent to $w_{1}, w_{2}$ and $w_{3}$.  By symmetry, we may assume that $\deg(w_{1}) \leq \deg(w_{2}) \leq \deg(w_{3})$. If $w_{1}$ is a $4^{-}$-vertex, then both $w_{2}$ and $w_{3}$ are $8^{+}$-vertices, and then the final charge of $w$ is at least $3 - 4 + 2 \times \frac{1}{2} = 0$. If $w_{1}$ is a $5$-vertex, then both $w_{2}$ and $w_{3}$ are $7^{+}$-vertices, and then the final charge of $w$ is at least $3 - 4 + 2 \times \frac{3}{7} + \frac{1}{5} > 0$. If $w_{1}$ is a $6^{+}$-vertex, then $w$ is adjacent to three $6^{+}$-vertices, and then its final charge is at least $3 - 4 + 3 \times \frac{1}{3} = 0$. 
\end{proof}

\begin{corollary}[Borodin and Ivanova \cite{MR3357780}]
Every triangle-free NPM has a $(4^{-}, 3, 7^{-})$-path, or a $(5, 3, 5)$-path or a $(5, 3, 6)$-path. 
\end{corollary}

\begin{theorem}\label{THMLast}
If $G$ is a triangle-free normal plane map, then it contains one of the following configurations: a $(3, 3, 3)$-path, a $(3, 3, 4)$-path, a $(3, 3, 5, 3)$-path, a $(4, 3, 4)$-path, a $(4, 3, 5)$-path, a $(5, 3, 5)$-path, a $(5, 3, 6)$-path and a $(3, 4, 3)$-path. 
\end{theorem}
\begin{proof}
Suppose that all the configurations are avoidable in a triangle-free normal plane map $G$. Further, we may assume that $G$ is connected. A $3$-vertex is called a $3^{*}$-vertex if it is adjacent to a $3$-vertex and a $5$-vertex. By the absence of $(3, 3, 5, 3)$-paths, if a $5$-vertex $w$ is adjacent to at least two $3$-vertices, then no adjacent $3$-vertex is a $3^{*}$-vertex. 

Initially, we assign each vertex $w$ the charge $\deg(w) - 4$ and each face $\alpha$ the charge $\deg(\alpha) - 4$. Thus, the sum of the initial charge is $-8$. We then design the following discharging rules to redistribute the charge, in the way that the final charge of each vertex and each face is nonnegative, which derives a contradiction. 
\begin{enumerate}[label = (R\arabic*)]
\item Each $6^{+}$-vertex $w$ sends $\frac{\deg(w) - 4}{\deg(w)}$ to each incident face. 
\item If a $5$-vertex $w$ is adjacent to a $3^{*}$-vertex, then $w$ sends $1$ to this $3^{*}$-vertex. 
\item If a $5$-vertex $w$ is not adjacent to any $3^{*}$-vertex, then $w$ sends $\frac{1}{5}$ to each adjacent vertex. 
\item Each $3$-vertex receives $\frac{1}{2}$ from each incident $5^{+}$-face. 
\item Let $w_{1}w_{2}w_{3}w_{4}$ be a $4$-face and $w_{3}$ be a $3$-vertex but not a $3^{*}$-vertex. If $w_{1}$ is a $3$-vertex and $w_{2}$ is a $4^{+}$-vertex, then the charge received from $w_{2}$ is evenly transferred to $w_{1}$ and $w_{3}$; and if $w_{1}$ and $w_{2}$ are all $4^{+}$-vertices, then the charge received from $w_{2}$ is transferred to $w_{3}$. 
\end{enumerate}

\bigskip
By the absence of $(3, 3, 3)$-paths, each face $\alpha$ is incident with at most $\left\lfloor \frac{2\deg(\alpha)}{3}\right\rfloor$ $3$-vertices. If $\alpha$ is a $6^{+}$-face, then its final charge is at least $\deg(\alpha) - 4 - \left\lfloor \frac{2\deg(\alpha)}{3}\right\rfloor \times \frac{1}{2} \geq \frac{2\deg(\alpha)}{3} - 4 \geq 0$. If a $5$-face is incident with at most two $3$-vertices, then its final charge is at least $5 - 4 - 2 \times \frac{1}{2} = 0$. Let $\alpha = w_{0}w_{1}w_{2}w_{3}w_{4}$ be a $5$-face with three incident $3$-vertices. By symmetry, we may assume that $w_{0}, w_{2}, w_{3}$ are $3$-vertices. By the absence of $(3, 3, 3)$-, $(3, 3, 4)$, $(3, 3, 5, 3)$-paths, it follows that $w_{1}$ and $w_{4}$ are all $6^{+}$-vertices, thus $\alpha$ receives at least $\frac{1}{3}$ from each of $w_{1}$ and $w_{4}$, and then its final charge is at least $5 - 4 + 2 \times \frac{1}{3} - 3 \times \frac{1}{2} = \frac{1}{6}$. Note that the initial charge of each $4$-face is zero. By (R5), the final charge of each $4$-face is nonnegative. 

Now, we consider the final charge of vertices. By (R1), the final charge of each $6^{+}$-vertex is nonnegative. If a $5$-vertex is adjacent to at least two $3$-vertices, then its final charge is at least $5 - 4 - 5 \times \frac{1}{5} = 0$; and if a $5$-vertex is adjacent to at most one $3$-vertex, then its final charge is at least $5 - 4 - 1 = 0$. Note that $4$-vertices are not involved in the discharging rules, thus its final charge is zero. So it suffices to consider $3$-vertices. 

Let $v$ be a $3$-vertex with three adjacent vertices $v_{1}, v_{2}$ and $v_{3}$. Without loss of generality, we may assume that $\deg(v_{1}) \leq \deg(v_{2}) \leq \deg(v_{3})$. Suppose that $v_{1}$ is a $3$-vertex. If $v$ is a $3^{*}$-vertex, then its final charge is at least $3 - 4 + 1 = 0$. So we may assume that $v$ is not a $3^{*}$-vertex. By the absence of $(3, 3, 3)$-, $(3, 3, 4)$-, $(3, 3, 5, 3)$-paths, both $v_{2}$ and $v_{3}$ are $6^{+}$-vertices. Recall that $v$ receives $\frac{1}{2}$ from each incident $5^{+}$-face. If $v_{1}vv_{2}w_{1}$ is a $4$-face, then $w_{1}$ is a $4^{+}$-vertex due to the absence of $(3, 3, 3)$-paths, and then $v_{2}$ contributes at least $\frac{1}{3}$ to $v$ via the $4$-face. In other words, $v$ receives at least $\frac{1}{3}$ via the face with face angle $v_{1}vv_{2}$ whenever the face is a $4^{+}$-face. Symmetrically, $v$ receives at least $\frac{1}{3}$ via the face with face angle $v_{1}vv_{3}$. If $v_{2}vv_{3}w_{2}$ is a $4$-face, then it receives at least $\frac{1}{6} + \frac{1}{6}$ from $v_{2}$ and $v_{3}$ via this $4$-face. In summary, $v$ receives at least $\frac{1}{3}$ from each incident face, thus its final charge is at least $3 - 4 + 3 \times \frac{1}{3} = 0$. 

Suppose that $v_{1}$ is a $4$-vertex. By the absence of $(4, 3, 4)$-paths and $(4, 3, 5)$-paths, it follows that $v_{2}$ and $v_{3}$ are $6^{+}$-vertices. By the absence of $(3, 4, 3)$-paths, if $v_{1}vv_{2}w_{1}$ is a $4$-face, then $w_{2}$ is a $4^{+}$-vertex, and then $v$ receives at least $\frac{1}{3}$ from $v_{2}$ via this $4$-face. Similar to the previous case that $v_{1}$ is a $3$-vertex, $v$ receives at least $\frac{1}{3}$ from each incident face in any situation, so the final charge is at least $3 - 4 + 3 \times \frac{1}{3} = 0$. 

Suppose that $v_{1}$ is a $5$-vertex. By the absence of $(5, 3, 5)$-paths and $(5, 3, 6)$-paths, both $v_{2}$ and $v_{3}$ are $7^{+}$-vertices. Note that $v_{1}$ is not adjacent to any $3^{*}$-vertex. The vertex $v$ receives $\frac{1}{2}$ or at least $\frac{3}{14} + \frac{3}{14}$ from the face with face angle $v_{2}vv_{3}$, receives $\frac{1}{2}$ or at least $\frac{3}{14}$ from each of the other two incident faces, and receives $\frac{1}{5}$ from $v_{1}$. Hence, the final charge of $v$ is at least $3 - 4 + \frac{3}{7} + 2 \times \frac{3}{14} + \frac{1}{5} > 0$. 

Suppose that $v_{1}$ is a $6^{+}$-vertex. The vertex $v$ receives $\frac{1}{2}$ or at least $\frac{1}{6} + \frac{1}{6}$ from each incident face, thus its final charge is at least $3 - 4 + 3 \times \frac{1}{3} = 0$. 
\end{proof}
\begin{corollary}[Borodin and Ivanova \cite{MR3357780}]\label{LC}
Every triangle-free normal plane map has a $(5^{-}, 3, 6^{-})$-path or a $(3, 4^{-}, 3)$-path. 
\end{corollary}

\section{Concluding remarks}
Note that we have different notations for the set of configurations. Some authors use $(a_{1}, a_{2}, \dots, a_{k})$ to denote the paths $v_{1}v_{2}\dots v_{\kappa}$ with $\deg(v_{1}) \leq a_{1}$, $\deg(v_{2}) \leq a_{2}, \dots, \deg(v_{\kappa}) \leq a_{\kappa}$; in our notations, this is denoted by $(a_{1}^{-}, a_{2}^{-}, \dots, a_{k}^{-})$. In \cite{MR3431391}, Jendrol' \etal defined the optimal unavoidable set $\mathcal{A}$ for a family $\mathcal{F}$. Formally, an unavoidable set $\mathcal{A}$ of paths is {\em optimal} for the family $\mathcal{F}$ if neither the type can be omitted from $\mathcal{A}$, nor any parameter $a_{i}$ of any type $(a_{1}^{-}, a_{2}^{-}, \dots, a_{k}^{-})$ from $\mathcal{A}$ can be decreased. In other words, they showed the optimality by providing an example with the configuration $(a_{1}, a_{2}, \dots, a_{\kappa})$ but no other configurations. In fact, $(a_{1}^{-}, a_{2}^{-}, \dots, a_{k}^{-})$ is a class of graphs, not just one graph. For example, $(4, 3, 6)$-path is a $(5^{-}, 3, 6^{-})$-path, and $(4, 3, 6)$-paths can be avoidable in the triangle-free normal plane map due to \autoref{THMLast}, but $(5^{-}, 3, 6^{-})$-path is an unavoidable configuration in the triangle-free normal plane map due to \autoref{LC}. 

In \cite{BorodinIvanovaKostochka2016}, Borodin \etal introduced an irreducible optimal unavoidable set of $3$-paths. Formally, for the set of types of $3$-paths $\mathcal{A} = \{(x_{1}^{-}, y_{1}^{-}, z_{1}^{-}), (x_{2}^{-}, y_{2}^{-}, z_{2}^{-}), \dots, (x_{\kappa}^{-}, y_{\kappa}^{-}, z_{\kappa}^{-})\}$, let $D(\mathcal{A})$ be a set of all triples $(x^{-}, y^{-}, z^{-})$ such that there exists a triplet $(x_{i}^{-}, y_{i}^{-}, z_{i}^{-})$ with $x \leq x_{i}, y \leq y_{i}$ and $z \leq y_{i}$. The optimal unavoidable set of types of 3-paths $\mathcal{A}$ is called {\em irreducible} if there is no other optimal unavoidable set of types of $3$-paths $\mathcal{A}'$ with $D(\mathcal{A}') \subsetneq D(\mathcal{A})$. Note that $D(\mathcal{A})$ is downwards according to the definition of irreducible optimal unavoidable set. We cannot guarantee that every member $A$ in $D(\mathcal{A})$ is unavoidable, that is, we cannot guarantee that there exists an example containing $A$ but no the other member in $D(\mathcal{A})$. So I think the best unavoidable set has the form $\{(x_{1}, y_{1}, z_{1}), (x_{2}, y_{2}, z_{2}), \dots, (x_{\kappa}, y_{\kappa}, z_{\kappa})\}$, such that no member in the set can be omitted and there exists an example for each member $A$ such that it contains $A$ but no the other member in the set. Please see a best unavoidable set in \autoref{MAD10/3}. Note that the best unavoidable set may be not unique.

\vskip 3mm \vspace{0.3cm} \noindent{\bf Acknowledgments.} This project was supported by the National Natural Science Foundation of China (11101125) and partially supported by the Fundamental Research Funds for Universities in Henan.

\end{document}